\documentclass[reqno]{amsart}

\newtheorem{theorem}{Theorem}
\newtheorem{lemma}[theorem]{Lemma}
\newtheorem{proposition}[theorem]{Proposition}
\newtheorem{corollary}[theorem]{Corollary}
\theoremstyle{definition}
\newtheorem{definition}[theorem]{Definition}

\theoremstyle{remark}

\numberwithin{equation}{section}
\usepackage{graphics}
\usepackage{mathtools}
\usepackage{bbm}
\usepackage[utf8]{inputenc} 
\usepackage[T1]{fontenc}
\usepackage{xcolor}
\usepackage{float}
\usepackage{amssymb}
\usepackage{srcltx}
\usepackage{enumerate}
\usepackage{esint}
\usepackage[shortlabels]{enumitem}
\usepackage{gensymb}
\usepackage{thmtools}

\usepackage{url}

\usepackage[colorlinks,plainpages=true,pdfpagelabels,hypertexnames=true,colorlinks=true,pdfstartview=FitV,linkcolor=blue,citecolor=red,urlcolor=black]{hyperref}

\everymath{\displaystyle}

\newcommand{\abs}[1]{\left\lvert#1\right\rvert}

\newcommand{\defeq}{\vcentcolon=}
\DeclarePairedDelimiterX{\inp}[2]{\langle}{\rangle}{#1, #2}

\DeclareMathOperator{\st}{{\bf st}}

\newcommand{\starint}{{\prescript{\ast}{}\int}}

\usepackage{graphicx}
\usepackage{amsmath}
\usepackage{tikz}
\usepackage{pgfplots}
\graphicspath{ {images/} }

\usepackage{tikz}
\usepackage{tikz-3dplot}
\usepackage{pgfplots}
\usepackage{thm-restate}

\pgfplotsset{%
    compat=1.8,
    compat/show suggested version=false,
}

\usetikzlibrary{calc,fadings,decorations.pathreplacing}
\usepackage[many]{tcolorbox}
\usepackage{wrapfig}

\begin{document}

\title[A nonstandard proof of de Finetti's Theorem]{A nonstandard proof of de Finetti's Theorem}

\author{Irfan Alam}
\address{Irfan Alam: Department of Mathematics, Louisiana State University, Baton Rouge, LA 70802, USA}
\email{irfanalamisi@gmail.com}
\urladdr{\url{http://www.math.lsu.edu/~ialam1}}




\begin{abstract}
We give a nonstandard analytic proof of de Finetti's theorem for an exchangeable sequence of Bernoulli random variables. The theorem postulates that such a sequence is uniquely representable as a mixture of iid sequences of Bernoulli random variables. We use combinatorial arguments to show that this probability distribution is induced by a hyperfinite sample mean. 
\end{abstract}

\maketitle
This paper presents an approach to de Finetti’s theorem based on nonstandard analysis, the necessary concepts of which are summarized in the Appendix. Throughout this paper, we have a fixed probability space $(\Omega, \mathcal{F}, \mathbb{P})$.

\begin{definition}\label{finite_exchangeable}
A finite collection $X_1, \ldots, X_n$ of random variables is said to be \textit{exchangeable} if for any permutation $\sigma \in S_n$, the random vectors $(X_1, \ldots, X_n)$ and $(X_{\sigma(1)}, \ldots, X_{\sigma(n)})$ have the same distribution. An infinite sequence $X_1, X_2, \ldots$ of random variables is said to be \textit{exchangeable} if any finite subcollection of the $X_i$ is exchangeable in the above sense.
\end{definition}

A well-known result of de Finetti says that a sequence of exchangeable Bernoulli random variables (that is, random variable taking values in $\{0, 1\}$) is conditionally independent given the value of a random parameter in $[0,1]$ (the parameter being sampled through a unique probability measure on $[0,1]$). More precisely, we may write de Finetti's theorem in the following form. 

\begin{theorem}[de Finetti]\label{de finetti theorem}
Let $X_1, X_2, \ldots$ be a sequence of exchangeable Bernoulli random variables. There exists a unique measure $\mu$ on the interval $[0,1]$ such that the following holds:
\begin{align}\label{de Finetti equation}
    \mathbb{P}(X_1 = e_1, \ldots, X_k = e_k) = \int_{[0,1]} p^{\sum_{j = 1}^k e_j}(1 - p)^{k - \sum_{j = 1}^k e_j} d\mu(p)
\end{align}
for any $k \in \mathbb{N}$ and $e_1, \ldots, e_k \in \{0,1\}$.
\end{theorem}

The integrand on the right side is the probability that $k$ iid Bernoulli($p$) random variables have the outcomes $e_1, \ldots, e_k$. In this sense, de Finetti's theorem expresses an exchangeable sequence of Bernoulli random variables as a \textit{mixture} of iid sequences of Bernoulli random variables. 

See de Finetti \cite{definetti1, definetti2} for the original formulations of this theorem. Aldous \cite{Aldous-book} and Kingman \cite{Kingman-survey} are good resources for an introduction to exchangeability and related topics. See Kirsch \cite{elementary-de-finetti} for a recent elementary proof of de Finetti's theorem.

We will give a nonstandard proof of Theorem \ref{de finetti theorem}. In nonstandard analytic language, the idea is that the measure $\mu$ will be shown to be induced by a hyperfinite sample mean $\frac{X_1 + \ldots + X_N}{N}$. A very brief introduction to nonstandard methods is provided in the appendix. We refer the reader to books such as \cite{Albeverio} and \cite{combinatorial} for more details.  

For the rest of this section, we fix an exchangeable sequence $X_1, X_2, \ldots$ of Bernoulli random variables. We also fix $k \in \mathbb{N}$ and $e_1, \ldots, e_k \in \{0, 1\}$. Taking $\alpha = \sum_{j = 1}^k e_j$ and writing the integral in \eqref{de Finetti equation} as an expectation in terms of a random variable $Y \sim \mu$, de Finetti's theorem may be restated as follows:
\begin{align}\label{de Finetti equation 2}
    \mathbb{P}(X_1 = e_1, \ldots, X_k = e_k) = \mathbb{E}_{\mu} (Y^\alpha(1 - Y)^{k - \alpha}).
\end{align}

Written this way, it is clear that any measure satisfying the conclusion of de Finetti's theorem must be unique. Indeed, taking $\alpha = k$ and varying $k$ through $\mathbb{N}$ in \eqref{de Finetti equation 2} shows that such a measure has a unique sequence of moments, which implies that they agree on expected values of continuous functions on $[0,1]$ (using the Weierstrass approximation theorem). 

Hence, it is enough to prove the existence of a probability measure on $[0,1]$ satisfying the conclusion of de Finetti's theorem. Toward that end, we will verify equation $\eqref{de Finetti equation 2}$ for a standard measure $\mu$ that is naturally induced by an appropriate Loeb measure. Fix $N > \mathbb{N}$ and define:

\begin{align}\label{definition of Y_N}
    Y_N = \frac{X_1 + \ldots + X_N}{N}.
\end{align}

Note that $Y_N$ takes values in $\left\{0, \frac{1}{N}, \ldots, \frac{N-1}{N}, \frac{N}{N} = 1\right\}$. Naively conditioning on the value of $Y_N$, we obtain the following:  

\begin{align}
 \mathbb{P}(X_1 = e_1, \ldots, &X_k = e_k) \nonumber \\
 = &\sum_{i = 0}^N {^*}\mathbb{P}\left(X_1 = e_1, \ldots, X_k = e_k \Big\vert Y_N = \frac{i}{N}\right) {^*}\mathbb{P}\left(Y_N = \frac{i}{N}\right). \label{naive conditioning} 
\end{align}

Note that we could have started the sum in \eqref{naive conditioning} at $i = \alpha$ since the conditional probabilities in this sum are zero for all $i < \alpha$. 

The random variable $Y_N$ induces an internal finitely additive internal probability measure $\mathbb{P}_N$ on ${^*}[0,1]$, which is supported on $\left\{0, \frac{1}{N}, \ldots, \frac{N-1}{N}, \frac{N}{N} = 1\right\}$, in the following way:

\begin{align}
    \mathbb{P}_N (B) = {^*}\mathbb{P}(Y_N \in B) \text{ for all ${^*}$-Borel sets } B \subseteq {^*}[0,1].
\end{align}
Consider the associated Loeb measure $L\mathbb{P}_N$. With $\mathcal{B}([0,1])$ denoting the Borel sigma algebra of $[0,1]$, define $\mu \colon \mathcal{B}([0,1]) \rightarrow [0,1]$ by:

\begin{align}\label{definition of mu}
    \mu(A) \defeq L\mathbb{P}_N(\st^{-1}(A)) \text{ for all Borel subsets } A \subseteq [0,1].
\end{align}

By Theorem \ref{appendix theorem}, $\mu$ is a well-defined Radon probability measure on $[0,1]$ such that the following holds:
\begin{align}\label{P_N and mu}
    {^*}\mathbb{E}_{\mathbb{P}_N} ({^*}f) \approx \mathbb{E}_{\mu}(f) \text{ for all bounded nonnegative } f\colon [0, 1] \rightarrow \mathbb{R}_{\geq 0}.
\end{align}

Consider the function $ f\colon [0, 1] \rightarrow \mathbb{R}_{\geq 0}$ defined by 

\begin{align}
    f(p) = p^{\alpha}(1 - p)^{k - \alpha} \text{ for all } p \in [0,1].
\end{align}

Noting the form of the right side in \eqref{de Finetti equation 2}, and using \eqref{naive conditioning} and \eqref{P_N and mu}, it is clear that we need the following to be true:

\begin{restatable}{theorem}{main}\label{conjecture}
We have
\begin{align}\label{what we need}
     &\sum_{i = 0}^N {^*}\mathbb{P}\left(X_1 = e_1, \ldots, X_k = e_k \Big\vert Y_N = \frac{i}{N}\right) {^*}\mathbb{P}\left(Y_N = \frac{i}{N}\right) \nonumber \\ 
     \approx &\sum_{i = 0}^N \left(\frac{i}{N} \right)^{\alpha} \left(1 - \frac{i}{N} \right)^{k - \alpha} {^*}\mathbb{P}\left(Y_N = \frac{i}{N}\right). 
\end{align}
\end{restatable}

The rest of this paper will build toward a proof of Theorem \ref{conjecture}. The strategy is to use the following simple fact from nonstandard analysis:

\begin{lemma}\label{hyperfinite lemma}
If $\alpha_j, \beta_j \in {^*}\mathbb{R}_{\geq 0}$ (where $j \in H$ for some hyperfinite set $H$) and $\frac{\alpha_j}{\beta_j} \approx 1$ for all $j \in H$, then 
\begin{align}
    \frac{\sum_{j \in H} \alpha_j}{\sum_{j \in H} \beta_j} \approx 1.
\end{align}
\end{lemma}
\begin{proof}
Let $H$, $\alpha_j$, and $\beta_j$ be as in the statement of the lemma. Note that $\alpha_j, \beta_j$ must all be strictly positive. For any real number $\epsilon \in \mathbb{R}_{>0}$, the condition that $\frac{\alpha_j}{\beta_j} \approx 1$ for all $j \in H$ implies that
$$1 - \epsilon < \frac{\alpha_j}{\beta_j} < 1 + \epsilon \text{ for all } j\in H.$$
Multiplying all sides of the above inequality by $\beta_j$, we have:
$$\beta_j - \epsilon \beta_j < {\alpha_j} < \beta_j + \epsilon \beta_j \text{ for all } j\in H.$$

Summing as $j$ varies over the hyperfinite set (in this step, we are also using transfer of a similar inequality for finite sums), we get:
\begin{align}
    \sum_{j \in H} (\beta_j - \epsilon \beta_j) < \sum_{j \in H} {\alpha_j} < \sum_{j \in H}(\beta_j + \epsilon \beta_j) \nonumber \\
    \Rightarrow (1 - \epsilon) \sum_{j \in H} \beta_j < \sum_{j \in H} \alpha_j < (1 + \epsilon) \sum_{j \in H} \beta_j. \label{alpha and beta}
\end{align}
Dividing all sides of \eqref{alpha and beta} by $\sum_{j \in H} \beta_j$ and noting that $\epsilon \in \mathbb{R}_{>0}$ was arbitrarily chosen completes the proof.
\end{proof}

For brevity in future computations, we define

\begin{align}\label{a_i}
    a_i &= {^*}\mathbb{P}\left(X_1 = e_1, \ldots, X_k = e_k \Big\vert Y_N = \frac{i}{N}\right) \\
    \text{ and } b_i &= \left(\frac{i}{N} \right)^{\alpha} \left(1 - \frac{i}{N} \right)^{k - \alpha} \text{ for all } i \in \{0, 1, 2, \ldots, N\}. \label{b_i}
\end{align}

Let us first try to understand the conditional probabilities $a_i$. As explained earlier, the $a_i$ are zero for $i < \alpha$. By summing over all possible cases, we have:

\begin{align}
    a_i &= {^*}\mathbb{P}\left(X_1 = e_1, \ldots, X_k = e_k \Big\vert Y_N = \frac{i}{N}\right) \nonumber \\
    &= \sum_{(u_1, \ldots, u_N) \in \mathcal{G}} {^*}\mathbb{P}\left(X_1 = u_1, \ldots, X_N = u_N \Big\vert X_1 + \ldots + X_N = i\right) \label{what we need 2},
\end{align}
where 
\begin{align*}
    \mathcal{G} &\defeq \left\{(u_1, \ldots u_N) \in \{0, 1\}^N: u_j = e_j \text{ for all } j \in \{1, \ldots, k\} \text{ and } \sum_{j = 1}^{N} u_j = i\right\}.
\end{align*}

It is clear that the internal cardinality of $\mathcal{G}$ is the number of ways of choosing $u_{k+1}, \ldots, u_N \in \{0,1\}$ such that $\sum_{j = k+1}^N u_j = i - \alpha$. By a simple counting argument, this yields:
\begin{align}\label{internal cardinality}
    \#(\mathcal{G}) = \binom{N - k}{i - \alpha}. 
\end{align}

Also, by the transfer of exchangeability of the $X_i$, it is clear that:
\begin{align}\label{exchangeability and conditioning}
   &{^*}\mathbb{P}\left(X_1 = u_1, \ldots, X_N = u_N \Big\vert X_1 + \ldots + X_N = i\right) \nonumber\\
    = &\frac{1}{\text{Number of ways of writing $i$ as a sum of $N$ zeroes and ones}}
\end{align}
for all $(u_1, \ldots, u_N) \in \mathcal{G}$. 

To see \eqref{exchangeability and conditioning}, first define $\mathcal{G}'$ as the set of those $(u_1, \ldots, u_N)$ such that $\sum_{j = 1}^N u_j = i$. Then exchangeability implies that 
\begin{align*}
    {^*}\mathbb{P}((X_1, \ldots, X_N) = \vec{u} ~&\vert~ X_1 + \ldots + X_N = i) \\
    = &{^*}\mathbb{P}((X_1, \ldots, X_N) = \vec{u}' \vert X_1 + \ldots + X_N = i) \text{ for all } \vec{u}, \vec{u}' \in \mathcal{G}'.
\end{align*}

Since the sum of ${^*}\mathbb{P}((X_1, \ldots, X_N) = \vec{u} ~\vert~ X_1 + \ldots + X_N = i)$ as $\vec{u}$ varies over $\mathcal{G}'$ is equal to one, it must be the case that 
\begin{align}\label{argument for G'}
    {^*}\mathbb{P}((X_1, \ldots, X_N) = \vec{u} ~\vert~ X_1 + \ldots + X_N = i) = \frac{1}{\#(\mathcal{G}')} \text{ for all } \vec{u} \in \mathcal{G}'.
\end{align}

In particular, since $\mathcal{G} \subseteq \mathcal{G}'$, equation \eqref{argument for G'} explains \eqref{exchangeability and conditioning}. Now, another simple counting argument shows that $\#(\mathcal{G}') = \binom{N}{i}$. Thus, \eqref{exchangeability and conditioning} becomes:

\begin{align}\label{exchangeability and conditioning 2}
{^*}\mathbb{P}\left(X_1 = u_1, \ldots, X_N = u_N \Big\vert X_1 + \ldots + X_N = i\right) &= \frac{1}{\binom{N}{i}}
\end{align}
for all $(u_1, \ldots, u_N) \in \mathcal{G}$. 

Using \eqref{exchangeability and conditioning 2} and \eqref{internal cardinality} in \eqref{what we need 2}, we obtain:

\begin{align}\label{what we need 3}
     a_i = \frac{\binom{N - k}{i - \alpha}}{\binom{N}{i}} \text{ for all } i \in \{1, \ldots N\},
\end{align}
where $\binom{N - k}{i - \alpha}$ is understood to be zero when $i < \alpha$.

Using \eqref{what we need 3}, we first prove Theorem \ref{conjecture} in a pathological case of zero probability (see Lemma \ref{pathological case}) that we will avoid afterward. Note that the conclusion of de Finetti's theorem implies that this pathological case can never happen, unless all the random variables $X_i$ are zero almost surely. However, since we are proving de Finetti's theorem, we have to take care of this case in a non-circular way, without using de Finetti's theorem. 

\begin{lemma}\label{pathological case}
Suppose $\mathbb{P}(X_1 = e_1, \ldots, X_k = e_k) = 0$. Then, \eqref{what we need} holds. 
\end{lemma}

\begin{proof}
Suppose $\mathbb{P}(X_1 = e_1, \ldots, X_k = e_k) = 0$. Suppose $i \geq \alpha$ and consider the event $\left\{Y_N = \frac{i}{N}\right\}$, which is the same as the event $\{X_1 + \ldots + X_N = i\}$. 

If the sum of $N$ zero-one random variables is $i \geq \alpha$ then some subcollection of $k$ such random variables must have had exactly $\alpha$ ones. Therefore, if $\mathcal{C}$ denotes the collection of all $k$ tuples of distinct indices from $\{1, \ldots, N\}$ (so that the internal cardinality $\#(\mathcal{C})$ is $\binom{N}{k}$), then we have 
$$\{X_1 + \ldots X_N = i\} \subseteq \bigcup_{(j_1, \ldots j_k) \in \mathcal{C}} \{X_{j_1} = e_1, \ldots, X_{j_k} = e_k\}.$$

By exchangeability, all events in the union on the right have the same probability as the event $\{X_1 = e_1, \ldots, X_k = e_k\}$, which is assumed to have probability zero. Since ${^*}\mathbb{P}$ is hyperfinitely subadditive, this implies that ${^*}\mathbb{P}(X_1 + \ldots + X_N = i) = 0$ whenever $i \geq \alpha$. Thus (using \eqref{what we need 3}), proving \eqref{what we need} is equivalent to proving the following:

\begin{align}\label{what we need 4}
     &\sum_{i = 0}^{\alpha - 1} \frac{\binom{N - k}{i - \alpha}}{\binom{N}{i}} \mathbb{P}\left(Y_N = \frac{i}{N}\right) \nonumber \\ 
     \approx &\sum_{i = 0}^{\alpha - 1} \left(\frac{i}{N} \right)^{\alpha} \left(1 - \frac{i}{N} \right)^{k - \alpha} \mathbb{P}\left(Y_N = \frac{i}{N}\right). 
\end{align}

But the left side of \eqref{what we need 4} is zero (as $\binom{N - k}{i - \alpha} = 0$ for $i < \alpha$), while the right side is an infinitesimal (being a finite sum of infinitesimals). This completes the proof.
\end{proof}

Also using \eqref{what we need 3}, we obtain the following result about the ratio of $a_i$ and $b_i$:

\begin{lemma}\label{ai over bi}
There exists a constant $r \approx 1$, such that for each $i \in {^*}\mathbb{N}_{>k}$, we have 
\begin{align}\label{first inequality lemma}
    \frac{a_i}{b_i} = \frac{i!}{(i - \alpha)!i^\alpha} \left(1 - \frac{1}{N - i} \right) \ldots \left(1 - \frac{k - \alpha - 1}{N - i} \right) r \leq r. 
\end{align}
\end{lemma}

\begin{proof}
From \eqref{what we need 3} and \eqref{b_i}, we obtain:
\begin{align*}
    \frac{a_i}{b_i} &= \frac{\frac{(N-k)(N-k-1)\ldots (N - k - (i - \alpha - 1))}{(i - \alpha)!}}{\frac{N(N-1) \ldots (N - (i-1))}{i!}\left( \frac{i}{N}\right)^{\alpha} \left(1 - \frac{i}{N} \right)^{k - \alpha}} \\
   =  &\frac{i!}{(i - \alpha)!i^\alpha} \frac{N^{k}}{(N - i)^{k - \alpha}} \frac{(N-k)(N-k-1)\ldots (N - k - (i - \alpha - 1))}{N(N-1) \ldots (N - (i-1))} \\
     = & \frac{i!}{(i - \alpha)!i^\alpha}  \frac{N^k (N-i) (N - (i+1)) \ldots N - (i + k - \alpha - 1)}{N(N-1)\ldots (N - (k-1)) (N - i)^{k - \alpha}} .
\end{align*}

Let 
\begin{align}\label{r definition}
    r \defeq &\frac{N^k}{N(N-1)\ldots (N - (k-1))} = \frac{1}{1 \left(1-\frac{1}{N}\right)\ldots \left(1 - \frac{k-1}{N}\right)} \approx 1. 
\end{align}

Thus the proof is complete in view of the following:
\begin{align*}
     &\frac{(N - i)(N - (i + 1)) \ldots (N - (i + k - \alpha - 1))}{(N - i)^{k - \alpha}} \\
     =  &1 \left(1 - \frac{1}{N - i} \right) \ldots \left(1 - \frac{k - \alpha - 1}{N - i} \right). 
\end{align*}
\end{proof}

\begin{lemma}\label{alpha > 1}
Suppose $\alpha \geq 1$. There is an $M_1 > \mathbb{N}$ such that $M_1 < N - \sqrt{N}$ and 
\begin{align*}
    \sum_{i = 0}^{M_1} a_i {^*}\mathbb{P}\left(Y_N = \frac{i}{N}\right) \approx 0 \text{ and }  \sum_{i = 0}^{M_1} b_i {^*}\mathbb{P}\left(Y_N = \frac{i}{N}\right) \approx 0.
\end{align*}
\end{lemma}

\begin{proof}
Fix any $M_1 > \mathbb{N}$ such that $M_1 < \min\{N^{\frac{1}{3}}, N - \sqrt{N}\}$. 

Note that $\sum_{i = 0}^k a_i$ is an infinitesimal. Hence, by \eqref{first inequality lemma}, it suffices to show that $\sum_{i = 0}^{M_1} b_i$ is an infinitesimal. Now, 

\begin{align*}
    \sum_{i = 0}^{M_1} b_i = \sum_{i = 0}^{M_1} \left(\frac{i}{N} \right)^{\alpha} \left( 1 - \frac{i}{N} \right)^{k - \alpha}  \leq \frac{{M_1}^{1 + \alpha}}{N^{\alpha}} < \frac{{N}^{\frac{1 + \alpha}{3}}}{N^\alpha}= \frac{1}{N^{\frac{2\alpha - 1}{3}}}.
\end{align*}

But the right side is an infinitesimal because $2\alpha > 1$ (as $\alpha \geq 1$ is assumed in the statement of the lemma). This completes the proof. 
\end{proof}

For the rest of this paper, let 
\begin{align}\label{definition of M_2}
    M_2 \defeq [N - \sqrt{N}] + 1,
\end{align}
where $[\cdot]$ is the greatest integer function.

\begin{corollary}\label{the only corollary}
For $i \in {^*}\mathbb{N}$ with $\mathbb{N} < i \leq M_2$, we have $\frac{a_i}{b_i} \approx 1$. 
\end{corollary}
\begin{proof}
Note that $\frac{i!}{(i - \alpha)!i^\alpha} = 1$ when $\alpha = 0, 1$. And for $\alpha \geq 2$, we have
\begin{align*}
  \frac{i!}{(i - \alpha)!i^\alpha}  = \left(1 - \frac{1}{i}\right) \ldots \left(1 - \frac{\alpha - 1}{i}\right) \approx 1 \text{ if } i > \mathbb{N}.
\end{align*}

Thus, we have:
\begin{align}\label{i factorial term}
     \frac{i!}{(i - \alpha)!i^\alpha} \approx 1 \text{ for all } i >  \mathbb{N}.
\end{align}

Now let $i$ be as in the statement of the corollary, i.e., $\mathbb{N} < i \leq M_2$. Then, $N - i \geq N - M_2 \geq \sqrt{N}$. Then,

\begin{align}\label{other part}
    \left(1 - \frac{1}{N - i} \right) \ldots \left(1 - \frac{k - \alpha - 1}{N - i} \right) \approx 1 \text{ as well.}
\end{align}

Using \eqref{i factorial term} and \eqref{other part} in \eqref{first inequality lemma} completes the proof.
\end{proof}

\begin{lemma}\label{alpha < k-1} Suppose $\alpha \leq (k -1)$. Then
\begin{align}\label{M_2}
     \sum_{i = M_2 + 1}^N a_i {^*}\mathbb{P}\left(Y_N = \frac{i}{N}\right) \approx 0 \text{ and }  \sum_{i = M_2 + 1}^N b_i {^*}\mathbb{P}\left(Y_N = \frac{i}{N}\right) \approx 0.
\end{align}
\end{lemma}

\begin{proof}
By \eqref{first inequality lemma}, it suffices to show that the second sum is an infinitesimal. Since the $b_i$ are all positive, we have the following estimate for the second term:
\begin{align*}
     \sum_{i = M_2 + 1}^N b_i {^*}\mathbb{P}\left(Y_N = \frac{i}{N}\right) &\leq \left(\max_{M_2 + 1 \leq i \leq N} b_i \right)  \sum_{i = M_2 + 1}^N {^*}\mathbb{P}\left(Y_N = \frac{i}{N}\right)  \\
     &\leq \max_{M_2 + 1 \leq i \leq N} \left(\frac{i}{N} \right)^{\alpha} \left(1 - \frac{i}{N} \right)^{k - \alpha} \\
      &\leq 1\cdot \left(1 - \frac{N - \sqrt{N}}{N} \right)^{k - \alpha} \\
      &= \left(\frac{1}{\sqrt{N}}\right)^{k - \alpha},
\end{align*}
where the last term is infinitesimal since $k - \alpha \geq 1$. 
\end{proof}

We are now in a position to prove Theorem \ref{conjecture}. We restate it here for convenience.

\main*

\begin{proof}
The case when $\alpha = 0$ is verified directly by plugging in $\alpha = 0$ to the formulae for $a_i$ and $b_i$ and using Lemma \ref{hyperfinite lemma}.

In the case when $\alpha = k$, using \eqref{what we need} and \eqref{b_i}, we get:
\begin{align*}
    \frac{a_i}{b_i} = \frac{\binom{N-k}{i-k}}{\binom{N}{i} \frac{i^k}{N^k}} = \frac{i!}{(i - k)!i^k} \frac{(N - k)! N^k}{N!}.
\end{align*}
This expression is infinitesimally close to $1$ whenever $i > \mathbb{N}$. Thus, Lemma \ref{alpha > 1} and Lemma \ref{hyperfinite lemma} complete the proof in this case. 

By Lemma \ref{pathological case}, we may also assume that $$\mathbb{P}(X_1 = e_1, \ldots, X_k = e_k) \neq 0.$$ Then using \eqref{naive conditioning}, we obtain
$$\sum_{i = 0}^N {^*}\mathbb{P}\left(X_1 = e_1, \ldots, X_k = e_k \Big\vert Y_N = \frac{i}{N}\right) {^*}\mathbb{P}\left(Y_N = \frac{i}{N}\right) \not\approx 0.$$

Thus, by Lemmas \ref{alpha > 1} and \ref{alpha < k-1}, we obtain:

\begin{gather*}
    \sum_{i = 0}^N {^*}\mathbb{P}\left(X_1 = e_1, \ldots, X_k = e_k \Big\vert Y_N = \frac{i}{N}\right) {^*}\mathbb{P}\left(Y_N = \frac{i}{N}\right) \\ \approx \sum_{i = M_1 + 1}^{M_2} a_i {^*}\mathbb{P}\left(Y_N = \frac{i}{N}\right),
\end{gather*}
and
\begin{gather*}
\sum_{i = 0}^N \left(\frac{i}{N} \right)^{\alpha} \left(1 - \frac{i}{N} \right)^{k - \alpha} {^*}\mathbb{P}\left(Y_N = \frac{i}{N}\right) \approx \sum_{i = M_1 + 1}^{M_2} b_i {^*}\mathbb{P}\left(Y_N = \frac{i}{N}\right).
\end{gather*}

Corollary \ref{the only corollary} together with Lemma \ref{hyperfinite lemma} now complete the proof in this case. 
\end{proof}

As $e_1, \ldots, e_k$ was an arbitrarily fixed finite sequence of zeros and ones, this proves de Finetti's Theorem \ref{de finetti theorem} using Theorem \ref{appendix theorem}.

We finish this section with a combinatorial-probabilistic interpretation of the proof. A main ingredient in the proof was Corollary \ref{the only corollary}. It shows that when $i$ is large (in the sense that it is hyperfinite) but not too large (in the sense that it is less than $M_2 = [N - \sqrt{N}] + 1$), then $\frac{a_i}{b_i}$ is infinitesimally close to $1$. Looking at the expressions \eqref{what we need 3} and \eqref{b_i} for $a_i$ and $b_i$ respectively, we can express the ratio as follows: 

\begin{align*}
    \frac{a_i}{b_i} = \frac{\binom{N - k}{i - \alpha} \binom{k}{\alpha}}{\binom{N}{i}} \cdot \frac{1}{\binom{k}{\alpha}\left(\frac{i}{N} \right)^{\alpha} \left(1 - \frac{i}{N} \right)^{k - \alpha}}.
\end{align*}

The first term on the right is an expression related to a certain hypergeometric random variable, while the second term is related to a certain binomial random variable. We can thus interpret Corollary \ref{the only corollary} as a statement about asymptotically approximating a hypergeometric random variable with a binomial random variable. More explicitly, Corollary \ref{the only corollary} says that as long as $i$ is neither too small not too large, then the probabilities $P_1$ and $P_2$ described by the following are very close to each other in the sense that $\frac{P_1}{P_2} \approx 1$:
\begin{enumerate}[(1)]
    \item\label{P1} Uniformly choose a random subset of size $i$ (here $i \geq \alpha$) from $\{1, \ldots, N\}$: thus all the $\binom{N}{i}$ subsets are equally likely to be chosen. Then $P_1$ is the probability that exactly $\alpha$ elements of $\{1, \ldots, k\}$ appear in this random subset of size $i$. 
    
    \item\label{P2} Take a coin with a probability of Heads being $\frac{i}{N}$. Then $P_2$ is the probability that exactly $\alpha$ Heads appear in $k$ independent tosses of this coin. 
\end{enumerate}

\appendix\section{Background from nonstandard analysis}
 \setcounter{theorem}{0}
     \renewcommand{\thetheorem}{\thesection.\arabic{theorem}}
This appendix provides an introduction to the nonstandard methods used in the paper. A significant part of this discussion is an abbreviated version of a similar introduction in Alam \cite{Alam}. Very roughly, a nonstandard extension of a set $S$ is a superset ${^*}S$ that preserves the ``first-order'' properties of $S$. That is, a property which is expressible using finitely many symbols without quantifying over any collections of subsets of $S$ is true if and only if the same property is true of ${^*}S$. This is called the \textbf{transfer principle} (or just \textit{transfer} for brevity). The set ${^*}S$ should contain, as a subset, ${^*}T$ for each $T \subseteq S$. Like subsets, other mathematical objects defined on $S$ also have extensions. So, a function $f\colon S \rightarrow T$ extends to a map ${^*}f\colon {^*}S \rightarrow {^*}T$, and relations on $S$ extend to relations on ${^*}S$. Hence there is a binary relation ${^*}<$ on ${^*}\mathbb{R}$, which we still denote by $<$ (an abuse of notation that we frequently make), and which is the same as the usual order when restricted to $\mathbb{R}$. 

In general, we fix a set $S$ consisting of atoms (that is, we view each element of $S$ as an ``individual'' without any structure, set-theoretic or otherwise), and extend what is called the superstructure $V(S)$ of $S$, which is defined inductively as follows (here, for any set $A$, the set $\mathcal{P}(A)$ denotes the power set of $A$):

\begin{equation}
    \label{superstructure} 
    \begin{array}{rcl} 
     V_0(S) &\defeq &  S, \\
     V_{n}(S) &\defeq & \mathcal{P}(V_{n-1}(S)) \text{ for all } n \in \mathbb{N}, \\
     V(S) &\defeq &  \bigcup_{n \in \mathbb{N} \cup \{0\}} V_n(S).
     \end{array}
\end{equation}

Choosing $S$ suitably, the superstructure $V(S)$ can be made to contain all mathematical objects relevant for a given theory. For example, if $\mathbb{R} \subseteq S$, then all collections of subsets of $\mathbb{R}$ live as objects in $V_2(S) \subseteq V(S)$. For a finite subset consisting of $k$ objects from $V_m(S)$, the ordered $k$-tuple of those objects is an element of $V_n(S)$ for some larger $n$; and hence the set of all $k$-tuples of objects in $V_m(S)$ lies as an object in $V_{n+1}(S)$. For example, if $x, y \in V_m(S)$, then the ordered pair $(x, y)$ is just the set $\{\{x\}, \{x,y\}\} \in V_{m+2}(S)$. Identifying functions and relations with their graphs, $V{(S)}$ also contains, if $\mathbb{R} \subseteq S$, all functions from $\mathbb{R}^n$ to $\mathbb{R}$, all relations on $\mathbb{R}^n$, etc., for all $n \in \mathbb{N}$. 

We extend the superstructure $V(S)$ via a \textit{nonstandard map},
$${^*}\colon V(S) \rightarrow V({^*}S),$$
which, by definition, is any map satisfying the following axioms:
\begin{enumerate}
    \item[(NS1)]\label{NS1} The transfer principle holds.
    \item[(NS2)]\label{NS2} ${^*}\alpha = \alpha$ for all $\alpha \in S$.
    \item[(NS3)]\label{NS3} $\{{^*}a: a \in A\} \subsetneq {^*}A$ for any infinite set $A \in V(S)$.
\end{enumerate}

A nonstandard map may not be unique. In practice, however, we fix a standard universe $V(S)$ and a nonstandard map ${^*}$. The reader is referred to \cite[Theorem 4.4.5, p. 268]{Model_Theory} or \cite[Chapter 1]{Albeverio} for a proof of the existence of a nonstandard map.

An object that belongs to ${^*}A$ for some $A \in V(S)$ is called \textit{internal}. A useful way to understand this concept is to think that internal objects are those that inherit properties from their standard counterparts by transfer. For instance, the internal subsets of ${^*}S$ are precisely the elements of ${^*}\mathcal{P}(S)$---a (reasonable) property satisfied by all elements of $\mathcal{P}(S)$ (that is, by all subsets of $S$) will thus transfer to all internal sets. As a consequence, the class of internal sets is closed under Boolean operations such as finite unions, finite intersections, etc. 

\begin{definition}
For a cardinal number $\kappa$, a nonstandard extension is called \textit{$\kappa$-saturated} if any collection of internal sets that has cardinality less than $\kappa$ and that has the finite intersection property has a non-empty intersection.
\end{definition}

 We will henceforth assume that the nonstandard extension we work with is sufficiently saturated (cf. \cite[Lemma 5.1.4, p. 294 and Exercise 5.1.21, p. 305]{Model_Theory}). 
 
An element in ${^*}\mathbb{R}$ will be called \textit{infinite} if it is larger than all elements in $\mathbb{R}$. Similarly an element in ${^*}\mathbb{R}$ will be called an \textit{infinitesimal} if its absolute value (that is, its image under the extension of the absolute value map) is smaller than all positive elements in $\mathbb{R}$. The set of non-infinite points in ${^*}\mathbb{R}$ is denoted by ${^*}\mathbb{R}_{\text{fin}}$. The next proposition shows that infinite (and infinitesimal) elements do exist in any sufficiently saturated nonstandard extension.

\begin{proposition}\label{existence of infinites}
${^*}\mathbb{R}$ contains infinite as well as infinitesimal elements.
\end{proposition} 
\begin{proof}
By saturation, the set $\cap_{n \in \mathbb{N}} \{x \in {^*}\mathbb{R}: x>n\}$ is empty. It is clear that any element in this set must be infinite. The multiplicative inverse of any infinite element is infinitesimal.
\end{proof}

The next result says that all legitimately nonstandard natural numbers (that is, those elements of ${^*}\mathbb{N}$ that are not elements of $\mathbb{N}$) are infinite.

\begin{proposition}\label{infinite natural}
Any $N \in {^*}\mathbb{N} \backslash \mathbb{N}$ is infinite. We express this by writing $N > \mathbb{N}$.
\end{proposition}
\begin{proof}
Let $N \in {^*}\mathbb{N} \backslash \mathbb{N}$. Suppose, if possible, that $N$ is finite. In particular, there exist elements of $\mathbb{N}$ that are larger than $N$. Thus the set $\{n \in \mathbb{N}: n > N\}$ is non-empty and hence has a smallest element, say $n_0$. By transfer of the fact that elements in $\mathbb{N}$ are at least one unit apart, we know that $n_0 - N \geq 1$. If $n_0 - N = 1$, then $N = n_0 - 1 \in \mathbb{N}$, a contradiction. Hence, we must have $n_0 - N \geq 2$ (by transfer of the fact that if the distance between two natural numbers is larger than one, then it is at least two). But then $n_0 - 1 \geq N + 1$ and $n_0 - 1 \in \mathbb{N}$, contradicting the minimality of $n_0$.    
\end{proof}

By transfer, all internal subsets of ${^*}\mathbb{R}$ have a least upper bound. Hence, in view of Proposition \ref{infinite natural}, the set $\mathbb{N}$ is not internal (as for any $N \in {^*}\mathbb{N} \backslash \mathbb{N}$, we have $N - 1 \in {^*}\mathbb{N} \backslash \mathbb{N}$ as well, and hence $\mathbb{N}$ does not have a least upper bound in ${^*}\mathbb{R}$).  We have seen several examples of internal sets and functions: ${^*}\mathbb{N}$, ${^*}\mathbb{R}$, ${^*}f$ (for any standard function $f$), etc. Unlike these examples, (NS3) guarantees the existence of internal objects that are not ${^*}\alpha$ for any $\alpha \in V(S)$. For instance, for any $N > \mathbb{N}$, the set $\{1, \ldots, N\}$ of the ``first $N$ nonstandard natural numbers'' is internal, yet it does not equal the nonstandard extension of any standard set. This set is rigorously defined as the initial segment of $N$ in ${^*}\mathbb{N}$. The fact that it is internal follows from the transfer of the following sentence:

\begin{align*}
    \forall n \in \mathbb{N} ~\exists! A \in \mathcal{P}(\mathbb{N}) ~[\forall x \in \mathbb{N} (x \in A \leftrightarrow x \leq n)].
\end{align*}

For a standard set $A$, let $\mathcal{P}_{\text{fin}}(A)$ denote the collection of finite subsets of $A$. There is a function $\#\colon \mathcal{P}_{\text{fin}}(A) \rightarrow \mathbb{N} \cup \{0\}$ that counts the number of elements in each finite subset. By transfer, we have a corresponding counting function ${^*}\#: {^*}\mathcal{P}_{\text{fin}}(A) \rightarrow {^*}\mathbb{N} \cup \{0\}$ (which we often still denote by $\#$ by an abuse of notation) that satisfies the same first order properties as the usual counting function (for example, it satisfies the inclusion-exclusion principle). The elements of ${^*}\mathcal{P}_{\text{fin}}(A)$ are called the \text{hyperfinite subsets} of ${^*}A$. Hyperfinite sets behave like finite sets even though they are not finite in the standard sense. For instance, an internal set $H$ is hyperfinite if and only if there is an $N \in {^*}\mathbb{N}$ and an internal bijection $f\colon H \rightarrow \{1, \ldots, N\}$. 

There is a ``sum function'' that takes any finite set of real numbers as an input and produces the sum of those real numbers. By transfer, we can thus abstractly make sense of ``hyperfinite sums'' (that is, the sum of hyperfinitely many nonstandard real numbers). For nonstandard real numbers $a_i$, this is the sense in which we interpret objects such as $\sum_{i = 1}^N a_i$ where $N \in {^*}\mathbb{N}$ (or in general, $\sum_{i \in H} a_i$, where $H$ is a hyperfinite set).

The next result says that one can think of a finite nonstandard real number $z$ as having a real part, and an infinitesimal part (in fact, this real part is just $\sup\{y \in \mathbb{R}: y \leq z\}$). See \cite[Theorem 2.10, p. 55]{Cutland_NATO} for a proof.
\begin{proposition}\label{standard part map}
For all $z \in {^*}\mathbb{R}_{\text{fin}}$, there is a unique $x \in \mathbb{R}$ (called the \textit{standard part} of $z$) such that $(z - x)$ is infinitesimal. We write $\st(z) = x$ or $z \approx x$. 
\end{proposition}

Note that, more generally, one can define the notion of standard parts for elements in the nonstandard extension of any Hausdorff space. In general, we will need a point to be \textit{nearstandard}, instead of finite, for it to have a standard part.

\begin{definition}
For a topological space $T$ and a point $z \in {^*}T$, we say that $z$ is nearstandard to $x$ if $z \in {^*}O$ for any open neighborhood $O$ of $x$.  
\end{definition}

If $T$ is a Hausdorff space, then a point $z \in {^*}T$ can be nearstandard to at most one point $x \in T$. In such a case, we write $x = \st(z)$ (also, $z \in \st^{-1}(x)$). Using this notation, we have the following useful characterization of continuity (see, for example, \cite[Proposition 1.3.3, p. 27]{Albeverio} for the one-dimensional case, with the higher dimensional case following a similar argument):

\begin{proposition}
Let $S$ and $T$ be Hausdorff spaces, and let $f\colon S \rightarrow T$ be a function. Then $f$ is continuous at $x \in S$ if and only if ${^*}f(\st^{-1}(x)) \subseteq \st^{-1}(f(x))$.
\end{proposition}

We shall also need the following characterization of compact spaces (see \cite[Proposition 2.1.6]{Albeverio}).
\begin{proposition}
A a topological space $T$ is compact if and only all points in ${^*}T$ are nearstandard.
\end{proposition}

The following consequence of saturation will be useful in the sequel (see \cite[Lemma 3.1.1, p. 64]{Albeverio} for a proof). 
\begin{proposition}\label{countable union}
A countable union of disjoint internal sets is internal if and only if all but finitely many of them are empty.
\end{proposition}

We now describe the concept of Loeb measures. Let $\mathfrak{X}$ be an internal set in a nonstandard universe ${^*}V(S)$. Let $\mathcal{A}$ be an \textit{internal algebra} on $\mathfrak{X}$, i.e., an internal set consisting of (internal) subsets of $\mathfrak{X}$ that is closed under complements and finite unions. Given an internal probability measure $\nu$ (that is, the internal map $\nu\colon \mathcal{A} \rightarrow {^*}\mathbb{R}_{\geq 0}$  satisfies $\nu(\mathfrak{X}) = 1$, and $\nu(A \cup B) = \nu(A) + \nu(B)$ whenever $A \cap B = \emptyset$), the map $\st(\nu) \colon \mathcal{A} \rightarrow {\mathbb{R}_{\geq 0}}$ is an ordinary finitely additive probability measure. By Proposition \ref{countable union}, it follows that $\st(\nu)$ satisfies the premises of Carath\'eodory Extension Theorem. By that theorem, it extends to a unique probability measure on $\sigma(\mathcal{A})$ (the smallest sigma algebra containing $\mathcal{A}$), whose completion is called the \textbf{Loeb measure} of $\nu$. The corresponding complete measure space $(\mathfrak{X}, L(\mathcal{A}), L\nu)$ is called the \textbf{Loeb space} of $(\mathfrak{X}, \mathcal{A}, \nu)$. Note that this construction could have been done with any finite internal measure $\nu$.

We will use the following simplification of \cite[Theorem 5.1, p. 105]{Ross_NATO} extensively:
\begin{proposition}\label{Loeb measurability}
Let $(\mathfrak{X}, L(\mathcal{A}), L\nu)$ be the Loeb probability space of $(\mathfrak{X}, \mathcal{A}, \nu)$. Suppose $F\colon \mathfrak{X} \rightarrow {^*}\mathbb{R}$ is an internal function that is measurable in the sense that $F^{-1}(B) \in \mathcal{A}$ for all $B \in {^*}\mathcal{B}(\mathbb{R})$ (where $\mathcal{B}(\mathbb{R})$ is the Borel $\sigma$-algebra on $\mathbb{R}$). If $F(x) \in {^*}\mathbb{R}_{\text{fin}}$ for $L\nu$-almost all $x \in \mathfrak{X}$, then $\st(F)$ is Loeb measurable (i.e., measurable as a map from $(\mathfrak{X}, L(\mathcal{A}))$ to $(\mathbb{R}, \mathcal{B}(\mathbb{R}))$).
\end{proposition}

For any probability measure $\nu$, there is an $\textit{integral operator}$ that takes certain functions (those in the space $L^1(\nu)$ of integrable real-valued functions on the underlying sample space of $\nu$) to their integrals with respect to $\nu$. By transfer, if $(\mathfrak{X}, \mathcal{A}, \nu)$ is an internal probability space, then we also have the associated space ${^*}L^1(\mathfrak{X}, \nu)$ of ${^*}$-integrable functions. For any  ${^*}$-integrable $F\colon \mathfrak{X} \rightarrow {^*}\mathbb{R}$, one then has $\starint_\mathfrak{X} F d\nu \in {^*}\mathbb{R}$, which we call the ${^*}$-\textit{integral} of $F$ over $(\mathfrak{X}, \mathcal{A}, \nu)$. 

The ${^*}$-integral on ${^*}L^1(\mathfrak{X})$ inherits many properties (an important one being linearity) from the ordinary integral by transfer. If $F$ is finite almost surely with respect to the corresponding Loeb measure, then $\st(F)$ is Loeb measurable by Proposition \ref{Loeb measurability}. In that case, it is interesting to study the relation between the ${^*}$-integral of $F$ and the Loeb integral of $\st(F)$. The following result covers this for a useful class of functions (see \cite[Theorem 6.2, p.110]{Ross_NATO} for a proof):
\begin{theorem}\label{S-integrable TFAE}
Suppose $(\mathfrak{X}, \mathcal{A}, \nu)$ is an internal probability space and $F \in {^*}L^1(\mathfrak{X}, \nu)$ is such that $L\nu(F \in {^*}\mathbb{R}_{\text{fin}}) = 1$. Then the following are equivalent:
\begin{enumerate}[(1)]
    \item\label{S1} $\starint_\mathfrak{X} \abs{F} d\nu \in {^*}\mathbb{R}_{\text{fin}}$, and 
    $$\st\left(\starint_\mathfrak{X} \abs{F} d\nu \right) = \lim_{m \rightarrow \infty} \st\left(\starint_\mathfrak{X}  \abs{F}\mathbbm{1}_{\{\abs{F} \leq m\}} d\nu \right).$$
    
    \item\label{S2} For every $M > \mathbb{N}$, we have $\st\left(\starint_\mathfrak{X} \abs{F} \mathbbm{1}_{\{\abs{F} > M\}}d \nu \right) = 0$.
    \item\label{S3} $\starint_\mathfrak{X} \abs{F} d\nu \in {^*}\mathbb{R}_{\text{fin}}$; and for any $A \in \mathcal{A}$ we have: 
    $$\nu(A) \approx 0 \Rightarrow \starint_\mathfrak{X} \abs{F} \mathbbm{1}_A d\nu \approx 0.$$
    \item\label{S4} $\st(F)$ is Loeb integrable, and $\st\left(\starint_\mathfrak{X} \abs{F} d\nu \right) = \int_\mathfrak{X} \abs{\st(F)} dL\nu$.
\end{enumerate}
\end{theorem}

A function satisfying the conditions in Theorem $\ref{S-integrable TFAE}$ is called $\mathit{S}$\textit{-integrable} on $(\mathfrak{X}, \mathcal{A}, \nu)$. Using this concept, we obtain the main result that was needed in this paper. The following result is applicable to more general situations (refer to the settings in Sections 3.4 and 3.5 of Albeverio et al. \cite{Albeverio}). However, we restrict to compact Hausdorff spaces and real-valued functions on them for convenience. 
\begin{theorem}\label{appendix theorem}
Let $S$ be a compact Hausdorff space. Suppose ${^*}\mathcal{B}(S)$ is the internal algebra of ${^*}$-Borel subsets of $S$. Let $\nu$ be an internal (finitely additive) probability measure on $({^*}S, {^*}\mathcal{B}(S))$. Let $L\nu$ be the associated Loeb measure. Define a map $\mu \colon \mathcal{B}(S) \rightarrow [0,1]$ by:

\begin{align}
    \mu(B) \defeq L\nu(\st^{-1}(B)) \text{ for all } B \in \mathcal{B}(S).
\end{align}

Then, we have:
\begin{enumerate}[(i)]
    \item\label{Radon} $\mu$ is a Radon probability measure.
    \item\label{integrals of continuous functions} For any nonnegative continuous function $f\colon S \rightarrow \mathbb{R}_{\geq 0}$, we have:
    \begin{align}\label{continuous integral equation}
        \starint_{{^*}S} {^*}f d\nu \approx \int_S f d\mu.
    \end{align}
    \end{enumerate}
\end{theorem}

\begin{proof}
Note that since $S$ is a compact space, we have $\st^{-1}(S) = {^*}S$. That $\mu$ is well-defined (that is, $\st^{-1}(B)$ is Loeb measurable for each $B \in \mathcal{B}(S)$) and is a Radon measure then follow from Proposition 3.4.5 and Corollary 3.4.3 in Albeverio et al. \cite[pp. 88-89]{Albeverio}. 

To see \ref{integrals of continuous functions}, let $f \colon S \rightarrow \mathbb{R}_{\geq 0}$ be a nonnegative function (which is automatically bounded, as the domain is a compact space). Since $f$ is bounded, it follows that $\st({^*}f)$ is Loeb measurable, satisfying the following (see Proposition \ref{Loeb measurability} and \ref{S2} $\Rightarrow$ \ref{S4} of Theorem \ref{S-integrable TFAE}):
\begin{align}\label{S-integrable}
      \starint_{{^*}S} {^*}f d\nu \approx \int_{{^*}S} \st({^*}f) dL\nu.
\end{align}

Also, with $\lambda$ denoting the one-dimensional Lebesgue measure, we have (since $\st({^*}f)$ is nonnegative):
\begin{align}\label{relation with Lebesgue}
    \int_{{^*}S} \st({^*}f) dL\nu &= \int_{(0, \infty)} L\nu \left\{x \in {^*}S : \st({^*}f(x)) > y \right\} d\lambda(y) \nonumber\\
    &= \int_{(0, \infty)} L\nu \left\{x \in {^*}S: f(\st(x)) > y \right\} d\lambda(y).
\end{align}

We used the nonstandard characterization of continuity (i.e., that $\st({^*}f(x)) = {^*}f(\st(x))$ for all nearstandard points $x \in {^*}S$, which in our case includes \text{all } $x\in {^*}S$ since $S$ is compact) to obtain \eqref{relation with Lebesgue} in the above.

For $y \in (0, \infty)$, let 
\begin{align*}
    A_y &\defeq \left\{x \in {^*}S: f(\st(x)) > y \right\}  \\
    \text{and } B_y &\defeq \{x \in S: f(x) > y\}.
\end{align*}

It is routine to verify that 
\begin{align}\label{A_y and B_y}
    A_y = \st^{-1}(B_y) \text{ for all } y \in (0, \infty). 
\end{align}

Thus, \eqref{relation with Lebesgue} becomes:
\begin{align}\label{integral with respect to mu}
    \int_{{^*}S} \st({^*}f) dL\nu &= \int_{(0, \infty)} L\nu(A_y) d\lambda(y) \nonumber \\
    &= \int_{(0, \infty)} L\nu(\st^{-1}(B_y)) d\lambda(y) \nonumber \\
    &= \int_{{^*}S} \st({^*}f) dL\nu \nonumber \\
    &= \int_{(0, \infty)} \mu(B_y) d\lambda(y) \nonumber \\
    &= \int_{S} f d\mu. 
\end{align}

Equations \eqref{S-integrable} and \eqref{integral with respect to mu} complete the proof. 
\end{proof}

\section*{Acknowledgments}
The author would like to thank Karl Mahlburg and Ambar Sengupta for numerous mathematical discussions. 
\bibliography{References}
\bibliographystyle{amsplain}
\end{document}